\documentclass[arxiv,reqno,twoside,a4paper,11pt]{amsart}
\usepackage{mltools}

\usepackage{mlmath62,mlthm10-REAR}

\usepackage{pgfplots}
\pgfplotsset{compat=1.15,ticks=none}
\tikzset{>=latex}

\usepackage{upref,amsmath,amssymb,amsthm,mathrsfs,amsfonts}
\usepackage{pgf,tikz, graphics, xypic, latexsym}
\usepackage{graphicx}
\usepackage[colorlinks=true,linkcolor=blue,citecolor=blue]{hyperref}

\usetikzlibrary{arrows}

\usepackage[scaled]{helvet} 
\usepackage{courier} 
\usepackage[mathbf]{euler}
\usepackage{mllocal} 

\usepackage{pgfplots}
\usetikzlibrary{arrows.meta}

\numberwithin{equation}{section}

\definecolor{qqwuqq}{rgb}{0,0,0}

\begin{document}

\date{\today}

\title[Zeta and Fredholm determinants]{Zeta and Fredholm 
determinants of 
	self-adjoint operators}

\author{Luiz Hartmann}
\address{Departamento de Matem\'atica, 
	Universidade Federal de S\~ao Carlos (UFSCar),
	Brazil}
\email{hartmann@dm.ufscar.br, luizhartmann@ufscar.br}
\urladdr{http://www.dm.ufscar.br/profs/hartmann}
\thanks{Partially support by FAPESP (2018/23202-1, 2021/09534-4) and we 
gratefully 
acknowledge the financial
	support of the Hausdorff Center for Mathematics, Bonn.}

\author{Matthias Lesch}
\address{Universit\"at Bonn,
	Germany}
\email{ml@matthiaslesch.de, lesch@math.uni-bonn.de}
\urladdr{www.matthiaslesch.de, www.math.uni-bonn.de/people/lesch}


	\maketitle

\begin{abstract} 
Let $L$ be a self-adjoint invertible operator in a Hilbert space such that 
$L^{-1}$ is $p$-summable. Under a certain discrete dimension spectrum 
assumption on $L$, we study the relation between the (regularized) Fredholm 
determinant, $\det_{p}(I+z\cdot L^{-1})$, on the one hand and the 
zeta regularized determinant, $\det_{\zeta}(L+z)$, on the other. One of 
the 
main results is the formula
\begin{equation*}
	\frac{\detz (L + z)}{\detz (L)} = \exp\Bl \sum_{j=1}^{p-1} 
	\frac{z^j}{j!} \cdot
	\frac{d^j}{dz^j} 
	\log \detz (L+z) |_{z=0} \Br \cdot \det_{p}(I + z \cdot L^{-1} ).
\end{equation*}
We show that the derivatives $\frac{d^j}{dz^j} 
\log \detz (L+z) |_{z=0}$ can be expressed in terms of (regularized) zeta 
values and heat trace coefficients of $L$. 
Furthermore, we give a general criterion in terms of the heat trace
coefficients (and which is, \eg fulfilled for large classes of elliptic 
operators) which guarantees that the constant term in the asymptotic 
expansion of the Fredholm determinant, $\log\det_{p} (I+z\cdot L^{-1})$, 
equals the zeta determinant of $L$.
\end{abstract}

\tableofcontents

\section{Introduction}
\label{s.Intro}

The classical theory of trace ideals (see \cite{Sim2005}) allows to 
generalize the 
notion of determinant to operators which differ from the identity by a 
$p$-summable operator. More concretely, let first $A$ be a trace class 
operator in a separable Hilbert space. Then 
\begin{equation*}
	\begin{aligned}
		\det (I + z\cdot A) &= \sum_{k=0}^{\infty} \tr(\Lambda^k A)\cdot z^k 
		=\prod_{n=1}^\infty (1+ z\cdot \lambda_n(A))
	\end{aligned}
\end{equation*}
is an entire function of $z$ with zeros exactly in $-1/\gl_{n}(A)$, where 
$\gl_{n}(A)$ are the non-zero eigenvalues of $A$ counted with multiplicity. 
This supports the intuition that $z\mapsto \det(I+z\cdot A)$ plays the role 
of 
the 
``characteristic polynomial'' of $A$. The value of $\det (I+z\cdot A)$ at $z$ 
is 
called Fredholm determinant of $I+z\cdot A$.

The notion of Fredholm determinant can be lifted to operators $A$ which are 
only in some Schatten class $\B^{p}(\calH)$ for some $p\geq 1$ (see 
\cite{Sim77}). 
Then $z\mapsto \det_{N+1}(I+z\cdot A)$, $N+1\geq p$, a regularized version of 
the Fredholm determinant, is an entire function of order at most $N+1$, again 
with zeros exactly in $-1/\gl_{n}(A)$. The regularization process to obtain 
$\det_{N+1}$ is the classical Weierstraß method of convergence generating 
factors which lead to canonical Weierstraß products of finite genus. 

A completely different notion of regularized determinants appears in the 
theory 
of elliptic operators on manifolds. Here one encounters {\it unbounded} 
self-adjoint operators $L$ with a meromorphic zeta function ($\zeta$-function)
\begin{equation*}
	\zeta(s;L) = \sum_{\gl\in\spec L\setminus\{0\}} \gl^{-s}
\end{equation*}
and, by some heuristics, on puts
\begin{equation*}
	\det_{\zeta}L:=\exp(-\zeta'(0;L))
\end{equation*}
and calls this the {\it zeta regularized determinant ($\zeta$-regularized 
determinant)} of $L$. This 
definition goes back to the celebrated work by Ray and Singer \cite{RS1971}. 
It has deep applications,\eg the Cheeger-M\"uller theorem on the 
analytic torsion.

The definition of $\zeta(s;L)$ extends easily to 
m-sectorial\footnote{\emph{An operator $L$ in a Hilbert space is m-sectorial 
if
its spectrum lies in a sector $z_0+\bigsetdef{z\in\C}{ |\arg z|<\delta }$
for some $\z_0>0$ and some $\delta<\pi$ and such that outside the
sector the resolvent satisfies an estimate $ \| (L - z)\ii\| \le C |z|\ii$
for $|z|\geq r>0$, for some $r>0$. The argument function is normalized such
that $\arg 1 = 0$. Cf. \cite[Sec. V.3.10]{Kat}}.} operators. We will only 
need this for the
\emph{normal} operator $L+z$ where $-z$ is in the resolvent set of $L$.
The $\zeta$-function 
\begin{equation}\label{eq.Intro.Zeta.L+z}
	\zeta(s;L+z) = \sum_{\gl\in\spec L\setminus\{0\}} (\gl+z)^{-s},
\end{equation}
a priori defined for $\Re s$ large and $-z$ in the resolvent set of $L$
is then a holomorphic function of two variables (see also Proposition
\ref{p.taylor.exp.logdetzeta} below).

It is then not hard to see that, say for an elliptic self-adjoint 
differential 
operator $L$ on a closed manifold, the function $z\mapsto \det_{\zeta}(L+z)$ 
extends to an entire function of finite order with zeros exactly in 
$-\gl_{n}(L)$, 
where $\gl_{n}(L)$ are the eigenvalues of $L$ counted with multiplicity. 

In \cite[Proposition 4.6]{Les:DRS} it was shown that for certain 
Sturm-Liouville operators one 
has 
\begin{equation}\label{eq-Lesch-MathNach}
	\frac{\det_{\zeta}(L+z)}{\det_{\zeta}(L)} = \det(I+z\cdot L^{-1}),
\end{equation}
that is there is a simple relation between the $\zeta$-regularized 
determinant 
on the one hand and the Fredholm determinant on the other hand.

One of the goals of this paper is to find the most general purely functional 
analytic setting which guarantees \Eqref{eq-Lesch-MathNach} to hold. The 
result is the
Theorem \ref{Theo-Main} in Section \ref{ss.RZDAE} below which generalizes 
\Eqref{eq-Lesch-MathNach} to 
\begin{equation}\label{Eq-Intro-2}
	\frac{\detz (L + z)}{\detz (L)} = \exp\Bl \sum_{j=1}^{p-1} 
	\frac{z^j}{j!} \cdot
	\frac{d^j}{dz^j} 
	\log \detz (L+z) |_{z=0} \Br \cdot \det_{p}(I + z \cdot L^{-1} ),
\end{equation}
Here $L$ is a self-adjoint bounded below operator in the Hilbert space 
$\calH$ with $p$-summable resolvent, $p\in \Z$, $p\geq 1$. Furthermore, 
roughly speaking it is needed that $\zeta(s;L)$ is meromorphic in $\C$ with 
$0$ being a regular point and $s=1,\ldots, p-1$ being at most simple poles. 
The numbers $\frac{d^j}{dz^j} \log\detz (L+z)|_{z=0}$ are ``global'' 
invariants of $L$ in the sense that they cannot be extracted from the heat 
expansion of $L$, respectively as poles of $\zeta(s;L)$. However, these 
numbers can be expressed in terms of (partie finie) zeta values and heat 
trace coefficients of $L$ (see Proposition \ref{p.taylor.exp.logdetzeta}).
 The 
non-trivial 
information of \Eqref{Eq-Intro-2}, however, is that
\begin{equation*}
	\log\detz(L+z) - \log\detz (L) - \log\det_p(I+z\cdot L^{-1})
\end{equation*}
is a polynomial of degree at most $p-1$, vanishing at $z=0$. For $p=1$ this 
specializes to \Eqref{eq-Lesch-MathNach}. The equality \eqref{Eq-Intro-2} 
contains 
another 
important information. Namely, if 
$\log\detz (L+z)$ has an asymptotic expansion as $\Re z\to \infty$ then so 
does 
$\log\det_p (I + z\cdot L^{-1})$.

It should have become clear that in general the $\zeta$-determinant 
is notoriously difficult to compute because it is a relatively rigid
spectral invariant which cannot be determined ``locally'' from, e.g.,
the resolvent or heat trace expansion. However, for \emph{some} operators
(see \cite{HLV17} and the references therein) explicit computations are
possible. Then the natural question arises whether one can say anything
meaningful about the $\zeta$-determinant of some perturbation of a
sufficiently nice operator. In this paper we intensively study the
behavior of the $\zeta$-determinant $\detz(L+z)$ as a function of
the resolvent parameter.

Equation \eqref{Eq-Intro-2}
leads us naturally to the second topic of this paper. Namely,
we would like to explore further the relation between the 
$\zeta$-re\-gu\-la\-ri\-zed determinant and the asymptotic expansion of 
$\det_{N+1}(I+z\cdot L^{-1})$ as $\Re z\to \infty$. To motivate this, we 
formally 
take the logarithmic derivative of $z\mapsto \det_{\zeta}(L+z)$ to obtain
\begin{equation*}
	\tr \Bl(L+z)^{-1}\Br;
\end{equation*}
this exists only if $L^{-1}$ is trace class. If $L^{-1}$ is in some Schatten 
class then differentiating often enough gives up to a constant factor
\begin{equation*}
	\tr \Bl(L+z)^{-N-1}\Br.
\end{equation*}
If the latter has an asymptotic expansion as $z\to \infty$ (as is the case 
for the elliptic differential operators as above) the integration should lead 
to an asymptotic expansion of $\log\det_{\zeta}(L+z)$, respectively 
$\log\det_{N+1}(I + z \cdot L^{-1})$, as $\Re z\to \infty$. For elliptic 
(pseudo)differential operators 
on compact manifolds (resp. 
elliptic boundary value problems if a boundary is present) the existence of 
the 
asymptotic expansion of $\log\detz(L+z)$ as $\Re z\to \infty$ can be proven 
by 
genuinely microlocal methods, see \cite[Appendix]{BFK92}.

We show in this paper that no further microlocal methods are needed. Rather 
the expansion of $\log\detz(L+z)$ as $\Re z\to \infty$ is equivalent to the 
resolvent trace expansion of $\tr (L+z)^{-p}$ as $\Re z\to \infty$ which is 
equivalent to the short time asymp\-totic expansion of the heat trace 
(Theorem 
\ref{Theorem-AsympExpLogZetaDet}, Section \ref{Sub-sec-Heat-Resolvent}). This 
contains the expansion result of \cite[Appendix]{BFK92} as special case.

Together with \eqref{Eq-Intro-2} we consequently show in a purely functional 
analytic setting that the existence of the heat, resp. resolvent trace 
expansion (plus a mild assumption on the absence of $\log$-terms which is 
made mostly for convenience) implies the existence of an asymptotic expansion 
as $\Re z\to\infty$ of $\log\det_{p}(I+z\cdot L^{-1})$.  In Corollary 
\ref{Cor-AEFD}, at the end, we give a precise expression of the coefficients 
of 
the expansion in terms of the heat expansion coefficients, resp. 
$\zeta$-function values. An interesting observation is that the constant term 
in the expansion of $\log\det_{p}(I+z\cdot L^{-1})$ equals $\log\detz L$. 
This generalizes the main result of \cite{Fri}. In the case of certain 
Sturm-Liouville operators, similar results were obtained before in \cite[p. 
3439]{HLV17} and in \cite{LesTol:DOD}. Recently, there has been quite some 
interest in concrete computations of $\zeta$-determinants and Fredholm 
determinant. We mention, \eg \cite{GKSpr,GKJFA,GKQAM} and the references 
therein.

\section{Generalities}
\label{s.Gen}

\subsection{Notation}
\label{Subsection-Notation}
We will denote by $\Z, \R, \C$ the integer, real, and complex
numbers, respectively, $\Z_+:=\{0,1,2,\ldots\}$ denotes the non-negative 
integers
and by $\N = \{1,2,3,\ldots\}$ we denote the natural numbers.

For an analytic function $f:\setdef{z\in \C}{0<|z-a|<\epsilon}\to \C$ in a 
punctured neighbourhood of $a\in \C$, we denote by $\Res^{m}_a f$ the 
coefficient of $(z-a)^{-m}$ in the Laurent expansion of $f$ about $a$, \ie 
the Laurent expansion reads   
\begin{equation}\label{Eq-DefRes}
	f(z) = \sum_{m=-\infty}^{\infty}  (\Res^{-m}_a f) \cdot  (z-a)^m.
\end{equation}
For the ordinary residue, we write $\Res_a f\equiv 
\Res_{a}^1f =a_{-1}$.

The letters $o, O$ will denote the Landau (small and big) $O$-symbols.

\subsection{Regularized integral and partie finie}
\label{ss.RIPF}
In this paper we will use freely the useful partie finie regularized integral,
\cf \cite[Sec.~2.1]{Les:OFT}, \cite{Les:DRS,LesTol:DOD},
\cite[Sec.~1]{LesVer11}, \cite[Sec.~1]{LesVer15}, \cite[Sec.~1]{HLV17}. However,
to fix some notation and for the convenience of the reader we present here the
cheat sheet version.

Let $f:(0,\infty)\to \C$ be a locally integrable function with asymptotic
expansions 
\begin{equation}\label{eq.gen.expansion}
\begin{split}
     f(x) &\sim \sum_{\ga \in \sA} \sum_{k=0}^{k_\ga^0} a_{\ga k}^0  x^{\ga} 
     \log^k x,
                   \quad \text{as } x\to 0,  \\
     f(x) &\sim \sum_{\gb \in \sB} \sum_{k=0}^{k_\gb^\infty} a_{\gb k}^\infty 
     x^{\gb} \log^k x,
                   \quad \text{as } x\to \infty.
\end{split} 
\end{equation}
Here, $\sA, \sB$ are discrete sets of complex numbers with the property that the
intersection with each finite vertical strip is finite. Moreover $+\infty$ is
the only accumulation point of $\Re \ga, \ga\in\sA$, and $-\infty$ is the only
accumulation point of $\Re \gb, \gb \in \sB$.  One could also work with partial
asymptotic expansions but for our purposes complete asymptotic expansions
suffice. The class of such functions $f$ satisfying \Eqref{eq.gen.expansion}
will be called ``regular functions with complete asymptotics on $\R_+$``.  
The sets $\sA, \sB$ may vary from function to
function. For convenience we adopt the convention that for $f$ the coefficients
$a_{\ga k}^{0}$, $a_{\ga k}^{\infty}$ are defined for all $\ga\in\C$ and all
$k\in\Z_+$ with the understanding that $a_{\ga k}^*=0$ for all but the countably
many $\ga\in \sA$, resp. $\sB$, and that for each fixed $\ga\in\C$ the
coefficients $a_{\ga k}^*$ vanish for all but finitely many $k$. The expansion
may then be written briefly as
\begin{equation*} 
    \sum_{\ga,k} a_{\ga k}^{0/\infty} \, x^\ga \log^k x.
\end{equation*}     

If $f: (0,\infty)\to \C$ is a regular function with complete asymptotics on 
$\R_+$ we
write
\begin{equation} 
  \begin{split} 
    \LIM\limits_{x\to 0} f(x)  := a_{00}^0, 
        \quad \LIM\limits_{x\to\infty} f(x):= a_{00}^\infty,
  \end{split} 
\end{equation} 
for the \emph{regularized limits} of $f$ as $x\to 0$, resp. $x\to\infty$. In
other words the regularized limit of $f$ as $x\to 0$ ($x\to\infty$) is the
coefficient of $x^0 \, \log^0 x$ in the asymptotic expansion of $f$ as $x\to 0$
($x\to \infty$).

If $f$ is a regular function with complete asymptotics on $\R_+$ then so is
$F(x):= \int_1^x f(x) dx$ and hence one puts
\begin{equation} 
  \regint\limits_0^\infty f(x) dx 
    := \LIM\limits_{x\to\infty} F(x) - \LIM\limits_{x\to 0} F(x)
     = \LIM_{R\to\infty}\LIM_{\eps\to 0} \int_\eps^R  f(x) dx,
\end{equation}
where the order of $\LIM$ after the second $=$ is inconsequential.  The
regularized integral is a convenient extension of the integral to the class of
regular functions with complete asymptotics. 

It has its peculiarities, though. We emphasize that changes of variables require
some care. For dilations one has the following formula
\cite[Lemma~2.1.4]{Les:OFT}:
\begin{equation}\label{eq.subst}
  \gl \regint_0^\infty f(\gl u) du
    = \regint_0^\infty f(x) dx
      + \sum_{k=0}^\infty \frac{ a_{-1, k}^\infty - a_{-1, k}^0}{k+1} 
      \log^{k+1} \gl.
\end{equation}
Another peculiarity is
\begin{equation}\label{eq.regzero}
  \regint_0^\infty x^\ga \cdot \log^k x \; dx =0,
\end{equation}
for \emph{all} $\ga\in\C, k\in\Z_+$ \cite[Eq.~2.1.12]{Les:OFT}.

\subsubsection{Mellin transform and relation to Hadamard's partie finie}
\label{sss.MTHPF}
For a regular function with complete asymptotics $f$ one defines for any $c>0$
the partial Mellin transforms by
\begin{equation} 
\begin{split} 
    (M_c^+f) (z)&:= \int_c^\infty x^{z-1} f(x) dx, \quad \Re z \ll 0, \\
    (M_c^-f) (z)&:= \int_0^c x^{z-1} f(x) dx, \quad \Re z \gg 0.
\end{split} 
\end{equation}
It turns out \cite[Sec.~2.1]{Les:OFT} that $M_c^{\pm}$ extends meromomorphically
to $\C$ and the \emph{Mellin transform} $(\sM f)(z):= (M_c^+f)(z) + (M_c^-f)(z)$
is well-defined independently of the choice of $c$. Furthermore, it is a
meromorphic function with poles of order at most $\max(k_\ga^\infty,
k_\ga^{0})+1$ in $-\ga$ \cite[Prop-Def. 2.1.2]{Les:OFT}.
There are explicit formulas for the principal part of $\sM f$ about $-\ga$
in terms of $a_{\ga,k}^{0/\infty}$, \cf \cite[Eq. 2.1.8 and
2.1.9]{Les:OFT}.
This may be viewed as a version of Watson's Lemma (see \cite[pg. 
133]{Wat1918})  for the Mellin transform. 

Outside the poles of $\sM f$ one now finds the convenient global representation
\begin{equation} 
    (\sM f)(z) = \regint_0^\infty x^{z-1} f(x) dx
\end{equation} 
for all $z\in \C$ except the set $-\sA \cup (-\sB)$.

The Mellin transform can be used for an alternative (equivalent) approach to the
regularized integral: namely, if $F$ is a meromorphic function in a 
neighbourhood
of $a\in \C$ one defines the \emph{finite part} of $F$ at $a$ by
\begin{equation} 
   \Pf_{z=a} F(z):= (\Pf F)(a) := \Res_{z=a} \Bl (z-a)\ii F(z) \Br,
\end{equation}
where the residue notation is explained in Section \ref{Subsection-Notation}.
In other words $(\Pf F)(a)$ is the constant term in the Laurent expansion of
$F(z)$ about $a$. Needless to say, if $F$ is regular at $a$ then 
$(\Pf F)(a) = F(a)$. Fortunately, it turns out that the finite part of the
Mellin transform at $1$ of a regular function with complete asymptotics gives
its regularized integral:

\begin{proposition}[{\cite[Prop.~2.1.7]{Les:OFT}}]
\label{t.Pf} Let $f:(0,\infty)\to \C$ be a regular function
with complete asymptotics. Then 
\begin{equation*}
     \regint_0^\infty f(x) dx = \Pf_{z=1} \sM f( z ) 
         = \Pf_{z=0} \int_0^1 x^z f(x) dx + \Pf_{z=0} \int_1^\infty x^z f(x) dx,
\end{equation*} 
where the integrals are a priori defined for $\Re z$ large, resp. small, and 
then
analytically continued to $\C$ (a neighbourhood of $z=0$ would suffice).
\end{proposition}

With this we have the formula, valid for all $z\in \C$,
\begin{equation}\label{eq.Pf.regint}  
     (\Pf \sM f)(z) = \regint_0^\infty x^{z-1} f(x) dx.
\end{equation} 

\subsubsection{Example} 
We illustrate the concepts introduced before for the derivatives of the
$\Gamma$--function. This will also be needed later. For future reference
we record the following facts about $\Gamma$:
\begin{align}
    \Gamma(s) &= \frac 1s -\gamma + O(s), &\frac 1{\Gamma(s)} = s
    +\gamma\cdot s^2+O(s^3),\text{ as } s\to 0,\label{eq.Gamma.facts.1} \\
    \Res_{-n}\Gamma & = \frac{(-1)^n}{n!}, & (\Pf \Gamma)(-n) =
    (\Res_{-n}\Gamma)\cdot \bl L_n -\gamma \br. 
    \label{eq.Gamma.facts.2} 
\end{align}
Here, $L_n:=\sum\limits_{j=1}^n \frac 1j$ and 
$\gamma:=\lim\limits_{n\to\infty} 
\Bl\sum\limits_{j=1}^n \frac 1j - \log
n\Br$ denotes the Euler-Mascheroni constant. 
Recall from Section \ref{Subsection-Notation} that $\Res_\xi f$ denotes the
residue of the analytic function $f$ about the isolated singularity $\xi$.
For $n\in \Z_+$, $n>0$, 
\begin{align}
\frac{1}{\Gamma(n+s)} &=\frac{1}{(n-1)!}+ 
\frac{1}{(n-1)!}\cdot(\gamma-L_{n-1})s + O(s^2),
	   \qquad\text{as}\qquad s\to 0,  \label{eq.Gamma.facts.3}\\
\frac{1}{\Gamma(-n+s)} &= (-1)^n n! \cdot s+(-1)^n n! \cdot(\gamma-L_n)s^2 + 
O(s^3),
\qquad\text{as}\qquad s\to 0. \label{eq.Gamma.facts.4}
\end{align}

We have
\begin{equation}\label{eq.Pf.deriv.Gamma} 
    (\Pf \pl^k \Gamma) (\ga ) = \regint_0^\infty x^{\ga -1}\, \log^k 
    x\,e^{-x} dx,
\end{equation} 
for all $\ga\in\C$. Clearly, the equality is valid for $\Re \ga>0$ as an
ordinary integral. Hence, by analytic continuation and \Eqref{eq.Pf.regint} it
is valid for all $\ga$. The change of variables formula will be used to show 
the next lemma.
\begin{lemma}\label{p.Gamma} 
  For $\ga,z\in\C$ with $\Re z>0$ we have
\begin{align*}
  z^\ga&\cdot \regint_0^\infty x^{\ga -1 }\, \log^k x\, e^{-xz} dx  \\
    &=  \regint_0^\infty x^{\ga -1 } \log^k( x/z) e^{-x} dx 
      +\delta_{-\ga,\Z_+} \frac{(-1)^{k+1} (\Res_{\ga} \Gamma )\log^{k+1} 
      z}{k+1} \\
    &= \sum_{j=0}^k  {k\choose j}(-1)^j \, \bl\Pf \pl^{k-j}\Gamma\br(\ga) \cdot \log^j z 
         +\delta_{-\ga,\Z_+} \frac{(-1)^{k+1} (\Res_{\ga} \Gamma )\log^{k+1} 
         z}{k+1},
\intertext{where}
     \delta_{-\ga,\Z_+} & = \begin{cases} 1,&\text{ if } -\ga\in\Z_+,\\
                                          0,&\text{ otherwise}.
                          \end{cases}
\end{align*}
\end{lemma}

In the proof we will need the combinatorial formula
\begin{equation}\label{eq.combinatorial} 
 \sum_{j=0}^k {k \choose j} \frac 1{j+1} (-a)^{k-j} a^{j+1} = \frac{(-1)^k a^{k+1}}{k+1}
\end{equation} 
which follows by noting that ${k\choose j}\frac{1}{j+1} = \frac 1{k+1} {k+1 \choose j+1}$
and applying the binomial theorem.

\begin{proof}
We apply the change of variables formula \Eqref{eq.subst} to the function 
$f(x) = (x/z)^{\ga-1} \log^k(x/z) e^{-x}$. Certainly this is a regular function
with complete asymptotics. Since $\Re z>0$ it decays exponentially as 
$x\to \infty$ and therefore we have $a_{\ga k}^\infty=0$. It remains to find the
terms $x^{-1} \log^k x$ in the asymptotic expansion as $x\to 0$, or in the
notation of \Eqref{eq.gen.expansion}, $a_{-1,k}^0$. The expansion as $x\to 0$ 
is
obtained from the Taylor expansion of $e^{-x}$. Thus if $-\ga\not\in\Z_+$ then
$a_{-1,k}^0=0$ as well. If $-\ga\in\Z_+$ then we find for the terms of the form
$x^{-1}\log^j x$ as $x\to 0$
\begin{align*}
   &(x/z)^{\ga-1} \sum_{j=0}^k {k\choose j} \bl - \log z\br^{k-j} \log^j x
   \frac{(-1)^{\ga}}{(-\ga)!} x^{-\ga} = \\
   &  x^{-1} z^{1-\ga}(\Res_{\ga}\Gamma) \sum_{j=0}^k {k\choose j} (-\log 
   z)^{k-j}
   \log^j x.
\end{align*}
Thus the correction term in the formula \eqref{eq.subst} equals
\begin{equation*} 
    (\Res_{\ga} \Gamma) z^{1-\ga} \frac{(-1)^{k+1} \log^{k+1} z}{k+1},
\end{equation*} 
where we have used the \Eqref{eq.Pf.deriv.Gamma} and 
\eqref{eq.combinatorial}.
The rest is now straightforward.
\end{proof}

\section{Operators with meromorphic $\zeta$-function}
\label{s.OMZ}

\subsection{Set up}
\label{Subsection-Setup}
For the trace ideals in the algebra of bounded operators on a separated 
Hilbert space we refer to \cite{Sim2005} as a standard reference.

We now describe the class of operators we are going to consider in the 
remainder
of this paper. Let $\calH$ be a separable complex Hilbert space and let 
$L:\sD(L)\subset \calH \to \calH$ be a self-adjoint operator, which is bounded
below. It is not a big loss of generality to assume for convenience that
\begin{equation}\label{eq-BoundedBelow}
  \langle L u , u \rangle \geq 0,\; u\in \dom(L)\subset \calH.
\end{equation}
We consider the
heat semigroup $e^{-tL}$ associated to $L$ for any $t \geq 0$.  Since $L  = 
L^\ast\geq 0$ we  have that $\spec L$ is contained in the 
nonnegative real axis.
We assume additionally that 
\begin{equation}\label{eq-Schatten}
  (L+I)^{-1}\in \B^{p}(\calH),
\end{equation}
for some $p\geq 1$. Here, $\B^{p}(\calH)$ denotes the Schatten class of
order $p$.
The condition \eqref{eq-Schatten} implies that the sum 
\begin{equation}\label{Eq-Zetafunction}
  \zeta(s;L):=\sum_{\gl\in\spec L\setminus\{0\}} \gl^{-s},
\end{equation}
converges absolutely and locally uniformly for $\Re(s)>p$. The function
$\zeta(s;L)$ is called the {\it zeta function ($\zeta$-function)} associated 
to $L$.
The Mellin transform of $e^{-\gl t}$,
\begin{equation*}
\int_{0}^{\infty} t^{s-1}  e^{-t\cdot\gl} dt = \Gamma(s) \cdot   
\gl^{-s},\quad \gl>0, 
\end{equation*}
implies the representation
\begin{equation}\label{Eq-ZetaHeat}
\begin{split} 
\zeta(s;L) &= \frac{1}{\Gamma(s)} 
        \int_{0}^\infty t^{s-1}\cdot \bl \tr ( e^{-t L} ) - \dim \ker L\br  dt, \\
        & = \frac{1}{\Gamma(s)} 
        \regint_{0}^\infty t^{s-1}\cdot \tr( e^{-t L} )  dt,
\end{split} 
\end{equation}
for $\Re(s)>p$. The second line follows from \Eqref{eq.regzero}.  Following
Connes-Moscovici (\cf \cite[Definition II.1]{CM} and 
\cite[Definition 2.1]{Les13}) we define:
\begin{definition}\label{p.DDS}
The operator $L$ is said to have \emph{discrete dimension spectrum}, if
$\zeta(s;L)$ extends to a meromorphic function in the complex plane $\C$ such
that on finite vertical strips one has a uniform estimate 
$|\Gamma(s)\cdot \zeta(s;L)|=O(|s|^{-N})$, when $\Im(s)\to \infty$, for each
$N\in \N$. We will denote by $\Sigma(L)$ the set of poles of the function
$\Gamma(s)\cdot \zeta(s;L)$.
\end{definition}

The discrete dimension spectrum assumption is, via the formula
\eqref{Eq-ZetaHeat}, in fact equivalent to a strengthened form of a complete
asymptotic expansion of the trace of the heat semigroup. To explain this,
we first note that by \Eqref{eq-BoundedBelow} and \eqref{eq-Schatten} one has
\begin{equation}\label{eq.heat.inf} 
  \tr(e^{-tL}) \sim \dim\ker L + O(e^{-t \gl_{\min}}), 
    \quad \text{as } t\to\infty,
\end{equation}
with $\gl_{\min}>0$ denoting the smallest positive eigenvalue of $L$.
Hence the asymptotic expansion as $t\to\infty$ in the sense of
\Eqref{eq.gen.expansion} is automatic. Now let us \emph{assume} that
\begin{equation}\label{eq.heat.zero}
\tr (e^{-t L}) \sim \sum_{\ga\in\sA} \sum_{k=0}^{k_\ga} A^{H}_{\ga k} \ 
t^{\alpha} \log^k t, \quad \text{as } t\to 0+,
\end{equation}
with a discrete set $\sA\subset \C$ such that the intersection with each finite
vertical strip is finite and such that $\Re \ga, \ga\in\sA$ has only $+\infty$
as accumulation point.
The assumption \eqref{eq.heat.zero} means in other words that the heat
semigroup is a regular function of $t$ with complete asymptotics in the sense
of Section \ref{ss.RIPF}.  Furthermore, by \Eqref{Eq-ZetaHeat},
$\Gamma(s)\cdot\zeta(s;L)$ is nothing but the Mellin-transform of the trace of
the heat semigroup. It then follows from the discussion in Section
\ref{sss.MTHPF} that $\zeta(s;L)$ extends meromorphically to $\C$ with poles in
$-\sA$. In fact, from Proposition \ref{t.Pf} we infer the universally valid 
formula
\begin{equation*} 
   \Bl \Pf \bl\Gamma(\cdot)\cdot\zeta(\cdot;L)\br \Br(s)
       = \regint_0^\infty t^{s-1} \tr(e^{-tL}) dt,\quad s\in\C,
\end{equation*}
resp., for all $s\in\C\setminus (-\sA)$ we have
\begin{equation*} 
   \zeta(s;L)
       = \frac{1}{\Gamma(s)}\regint_0^\infty t^{s-1} \tr(e^{-tL}) dt.
\end{equation*}
As a word of warning, we emphasize that one cannot expect this to hold for the
finite parts in the poles, \cf the discussion in the next section below.

For $\ga\in\sA$ the principal part of $\zeta(s;L)$
about $-\ga$ is determined by
\begin{equation}\label{eq.zeta.1}
	\frac{1}{\Gamma(s)}\sum_{k=0}^{k_\alpha} (-1)^k k!\ 
	\frac{A^{H'}_{\alpha k}}{(s+\alpha)^{k+1}},
\end{equation}
where (see \eg \cite[Sec.~2]{BL99}).
\begin{equation}\label{eq.AH}
	A^{H'}_{\alpha k} := 
\left\{ 
	\begin{aligned}
		&A^{H}_{\alpha k},   &(\ga,k)\not = (0,0),\\
	&A^{H'}_{\alpha k} = A^H_{00} - \dim \ker L, &(\ga,k) = (0,0).
	\end{aligned}
\right.
\end{equation}
As a consequence, if $\alpha \not \in \Z_+$ the order of 
the pole $-\alpha$ is $k_\alpha+1$ with residue
\begin{equation}\label{eq.zeta.2}
{\rm Res}^{k_\alpha+1}_{-\ga} \zeta(\cdot;L) = (-1)^{k(\alpha)} 
\frac{A^H_{\alpha 
k_\alpha} (k_\alpha)!}{\Gamma(-\alpha)},
\end{equation}
and if $\alpha \in \Z_+$ then, since $\Gamma(s)^{-1}$ has a simple zero at
$-\alpha$, the order of the pole $-\alpha$ is $k_\alpha$ with residue
\begin{equation}\label{eq.zeta.3}
    \Res^{k_\alpha}_{-\alpha} \zeta(\cdot;L) = (-1)^{k_\alpha+\alpha} 
\alpha! (k_\alpha)!A^{H'}_{\alpha k_\alpha}.
\end{equation}
Recall that $\Res^k_a f$ denotes the coefficient of $(s-a)^{-k}$ in the 
Laurent
expansion of $f$ about a, see Section \ref{Subsection-Notation}. 

While the complete asymptotic expansion implies that $\zeta(s;L)$ has
a meromorphic extension to the whole plane, in order to obtain the expansion
from the meromorphic $\zeta$-function, one needs the finite vertical strip
decay assumption in order to be able to shift vertical integration contours
from the inverse Mellin transform (\cf \cite[Sec.~2]{BL99}). This condition,
being slightly stronger than just meromorphicity, is reflected in a stronger
asymptotic expansion condition. The result, being a standard application
of complex analysis, reads as follows (see \eg \cite[Lemma  2.2]{BL99} and
the references therein).
\begin{proposition}\label{Prop-DDS-HAE} Let $L$ be a bounded below
discrete operator in a Hilbert space with resolvent of $p$-Schatten class.
The operator $L$ has discrete dimension spectrum if, and only if, 
$\tr(e^{-t L})$ has an asymptotic expansion 
\eqref{eq.heat.zero} that can be differentiated, \ie for $N\in \Z_+$, 
$K>0$, 
we have
\begin{equation}\label{Eq-Heat-Asymptotic-Exp-Prop-1}
\left| \b_t^N \Bl \tr(e^{-t L}) - \sum_{\Re \alpha\leq 
N+K}\sum_{k=0}^{k_\alpha} A^{H}_{\alpha k}\ t^{\alpha}\log^k t 
\Br\right|\leq C_{N,K} t^{K}, \qquad \; {\text{as}}\qquad t\to 
0+.
\end{equation}
\end{proposition}

From the asymptotic expansion of the trace of the heat semigroup one can 
derive various other trace expansions. Well-known is the relation to
the resolvent expansion which we will also briefly review in Section
\ref{Sub-sec-Heat-Resolvent}
below. Less well-known is probably the expansion of the 
\emph{zeta determinant ($\zeta$-determinant)}
which we are going to explain first.

\subsection{The $\zeta$-determinant and its expansion
in terms of the resolvent parameter}
\label{ss.ZDER}

\subsubsection{Standing assumptions}\label{sss.ZDER.SA}
During this subsection, unless otherwise said, $L$ denotes a discrete
bounded below (\Eqref{eq-BoundedBelow})  self-adjoint operator in the Hilbert 
space $\calH$ whose
resolvent lies in the $p^{th}$-Schatten class.  Furthermore, we assume
\Eqref{eq.heat.zero} to hold and from now on we write $A_{\alpha 
		k}^H (L)$ to indicate the dependence of the operator $L$ on the heat 
		coefficients; recall that \Eqref{eq.heat.inf} is automatic.
By \Eqref{eq.zeta.3}, $0$ is a pole of $\zeta(s;L)$ of order $k_0 := k(0)$.  
In this
section, we would like to ensure that $\zeta(s;L+z)$ is regular at $s=0$
for all $\Re z>0$ and therefore in this section we assume in addition
that $k_{-n}=0$ for all $n\in\Z_+$.  In view of 
\Eqref{eq.zeta.1} and
\eqref{eq.zeta.2} this is equivalent to $A_{-n,k}^H(L)=0$ for 
all $n\in\Z_+$ and all $k>0$. In terms of $\zeta(s;L)$ this means
that $\zeta(s;L)$ is regular at $0$ and that in the positive integers
there are at most poles of at most order $1$.

\subsubsection{The $\zeta$-determinant of $L+z$}

Under our standing assumptions $\zeta(s;L)$ is regular at $s=0$.

\begin{definition} The $\zeta$-regularized determinant of $L$ is defined by
\[
	\det_{\zeta}(L):=\exp(-\zeta'(0;L)).
\]
\end{definition}

Our first result relates the $\zeta$--determinant of $L$ to the heat trace.

\begin{proposition}\label{Prop-Zetadeterminant} 
	Let $L$ satisfy the Standing assumptions \ref{sss.ZDER.SA}. 
Then $\zeta(s;L)$ is regular at $s=0$ and one has 
\begin{align}
     \zeta(0;L) & = A_{00}^H(L) - \dim \ker 
     L,\label{eq.zeta.det.asymp.heat.exp.1}\\
     -\log\detz(L) &= \gamma \cdot \Bl A_{00}^H(L)-\dim \ker L \Br + 
     \regint_0^\infty t\ii 
         \tr( e^{-t L} )dt.\label{eq.zeta.det.asymp.heat.exp.2}
\end{align}
Here, $\gamma$ denotes the Euler-Mascheroni constant 
(see \Eqref{eq.Gamma.facts.1}, \eqref{eq.Gamma.facts.2}) 
and $A_{00}^H(L)$ denotes the coefficient of $t^0 \log^0 t$ in the
asymptotic expansion of $\tr(e^{-tL} )$ as $t\to 0+$, see 
\Eqref{eq.heat.zero} 
and \eqref{eq.AH}.
\end{proposition}

While \Eqref{eq.zeta.det.asymp.heat.exp.1} has been known for a long time, 
\Eqref{eq.zeta.det.asymp.heat.exp.2} is more subtle. At least for Dirac type 
operators,
\Eqref{eq.zeta.det.asymp.heat.exp.2} 
was already known to K. P. Wojciechowski, \cf \cite[Eq. (3.5) and 
(3.6)]{WOJ1999}. These formulas are the key to a simple derivation of the 
asymptotic
expansion of $\zeta(0;L+z)$ and $\log\detz(L+z)$ as 
$\Re z\to\infty$.
The latter will lead to yet another equivalent characterization of
the discrete dimension spectrum assumption. It is on the same footing
as the resolvent expansion.
\begin{proof} As $s\to 0$ we have for the Mellin transform
\begin{equation*}
   \regint_0^\infty t^{s-1} \tr(e^{-tL}) dt = \frac{A_{00}^H(L)-\dim\ker L}{s}
        + \regint_0^\infty t\ii \tr(e^{-tL}) dt+ O(s).
\end{equation*}
We use, as before,  $A^{H'}_{00}(L):= A_{00}^{H}(L)-\dim\ker L$. By 
Section
\ref{sss.MTHPF} and \Eqref{eq.Gamma.facts.1} we obtain as $s\to 0$
\begin{align*}
    \zeta(s;L) &= \bl s + \gamma\cdot s^2+ O(s^3) \br \cdot
          \Bl \frac{ A^{H'}_{00}(L)}{s} + \regint_0^\infty t\ii
          \tr(e^{-tL}) dt+ O(s) \Br,\\
     &= A^{H'}_{00}(L) + \Bl \gamma A^{H'}_{00}(L) + \regint_0^\infty t\ii 
     \tr(e^{-tL}) dt \Br \cdot s
     + O(s^2)
\end{align*}
and the claim follows.
\end{proof}

We will later need the behavior of $\zeta(s;L)$ near $n>0$, $n\in \Z_+$. 
Under the Standing assumptions  \ref{sss.ZDER.SA}, we apply the same 
strategy as in the proof of the previous proposition and obtain for the poles 
of $\zeta(s;L)$
\begin{equation*}
	\zeta(s;L) \sim \sum_{n=1}^{\infty} \frac{A_{-n,0}^H}{\Gamma(s)} 
	\cdot	\frac{1}{s-n} + {\rm Entire}(s).
\end{equation*}
Hence for an integer $n>0$, 
\begin{equation}\label{eq.zeta.near.n}
	\zeta(s+n;L) \sim  \frac{A^H_{-n,0}}{(n-1)!} \cdot \frac{1}{s} + 
	\Pf_{s=n} 
	\zeta(s;L)+O(s) , \qquad\text{as}\qquad s\to0.	
\end{equation}

\begin{lemma}\label{p.zetanull}
Let $L$ satisfy the Standing assumptions \ref{sss.ZDER.SA}. 
Then we have for $\Re z>0$
\[
  \zeta(0;L+z)=A_{00}^H(L+z) = \sum_{n=0}^\infty (\Res_{-n}\Gamma)\cdot
  A_{-n,0}^H(L)\cdot z^n,
\]
Furthermore, the $\zeta$-function of $L+z$ is also regular at $0$ and
it has at most simple poles in the positive integers. In other words,
for all $\Re z>0$ the operator $L+z$ satisfies the Standing assumptions
\ref{sss.ZDER.SA}.
\end{lemma}
Note that the previous sum is finite.

\begin{proof} Let $\Re z>0$. Then $\tr(e^{-t(L+z)}) = O( e^{-tz})$, as $t\to\infty$.
As $t\to 0+$ we have
\begin{align*}
  \tr(e^{-t(L+z)}) &= e^{-tz} \cdot \tr(e^{-tL})\\
    & \sim \sum_{n=0}^\infty \frac{(-z)^n}{n!} t^n \cdot 
       \sum_{\ga , k} A_{\ga k}^H(L) \cdot t^\ga \log^k t\\
    & \sim \sum_{\ga} \sum_{n=0}^\infty \sum_{k} (\Res_{-n}\Gamma)\cdot 
        A_{\ga-n, k}^H(L)\cdot z^n\cdot t^{\ga} \log^k t.
\end{align*}
This gives in fact the general formula
\[
    A_{\ga  k}^H(L+z) =  \sum_{n=0}^\infty (\Res_{-n}\Gamma)\cdot 
    A_{\ga-n,k}^H(L)\cdot z^n.
\]
From this we read off the claim about the constant term as well as the fact
that since $A_{\ga k}^H(L)=0$ for $-\ga\in\Z_+, k>0$, this is also true for
the $A_{\ga k}^H(L+z)$.
\end{proof}

\begin{lemma}\label{p.regint.expand}
Let $L$ satisfy the Standing assumptions \ref{sss.ZDER.SA}. 
Then as $\Re z\to\infty$
\[
  \regint_0^\infty t\ii \tr (e^{-t(L+z)}) dt
   \sim \sum_{\ga,k} A_{\ga k}^H(L) \regint_0^\infty e^{-tz} \cdot t^{\ga -1} 
   \log^k t \, dt
     +O( e^{-\Re z} ).
\]
\end{lemma}
\begin{proof} We split the integral into $\regint_0^1$ and $\regint_1^\infty$.
Since $\tr( e^{-t(L+z)} ) = O( e^{-\Re z} )$, for $t \geq 1$ the contribution
of $\regint_1^\infty$ is $O( e^{-\Re z} )$ as well, and hence it contributes
only to the error term of the claimed expansion.

Next we fix a large $N$ and write
\[
   \tr(e^{-tL}) = \sum_{\Re \ga \le N}\sum_{ k=0}^{ k_\ga} 
   A_{\ga k}^H(L)\cdot t^\ga \log^k t + O(t^N),
\]
as $t\to 0+$.  The contribution of the remainder to $\regint_0^1$ is 
$O(z^{-N})$. So it remains to look at each individual summand which
contributes
\[
  \regint_0^1 e^{-tz} \cdot t^{\ga-1} \log^k t\, dt =
    \regint_0^\infty e^{-tz} \cdot t^{\ga-1} \log^k t\,  dt 
  - \int_1^\infty e^{-tz} \cdot t^{\ga-1}  \log^k t\, dt.
\]
The contribution of the last summand is again of exponential decay in $z$.
Thus
we reach the conclusion.
\end{proof}

Combining Lemmas \ref{p.Gamma}, \ref{p.zetanull}, and \ref{p.regint.expand}
we arrive at

\begin{theorem}\label{Theorem-AsympExpLogZetaDet} 
	Let $L$ be a discrete self-adjoint bounded below operator in
 the Hilbert space $\calH$ with resolvent of $p^{th}$-Schatten class. 
 Moreover,
 assume that the trace of the heat semigroup has an asymptotic expansion
 \eqref{eq.heat.zero} with $A_{\ga k}^H(L)=0$ for $\ga\leq 0$ and 
 $k>0$\footnote{i.e. the
 Standing assumptions \ref{sss.ZDER.SA} are fulfilled.}. Then 
 $\log\detz(L+z)$ has an asymptotic expansion as $\Re z\to \infty$
\begin{align}
  \log&\detz(L+z) \sim_{\Re z \to \infty} \nonumber\\
        &-\sum_{\stackrel{\ga, k}{-\ga\not \in \Z_+}} (-1)^{k} \sum_{j\geq k} 
        A_{\ga j}^H(L)\cdot  {j\choose k}\cdot
        \bl \Pf \pl^{j-k}\Gamma\br(\ga) \cdot z^{-\ga}\log^k z 
        \label{eq.asymp.log.zeta.det.1}\\
        &+ A_{00}^H(L) \cdot \log z \label{eq.asymp.log.zeta.det.2}\\
        &+ \sum_{n=1}^\infty A_{-n \;0}^H(L)\cdot  (\Res_{-n} 
        \Gamma)\cdot \Bl z^n\log z - L_n \cdot z^n \Br, 
        \label{eq.asymp.log.zeta.det.3}
\end{align}
where $L_n = \sum_{j=1}^n \frac 1j$.
\end{theorem}
Note that the coefficient of $z^0\log^0 z$ in this expansion vanishes.
Thus we obtain a different proof of \cite[Lemma 2.2]{HLV17}.

For (general) pseudodifferential operators on a compact manifold
(with or without boundary) the result was proved by genuinely pseudodifferential
methods in \cite[Appendix]{BFK92}. Our standing assumption is in general
only fulfilled for \emph{differential operators}. Our result
shows, however, that even if there are no $\log$-terms in the heat expansion,
nevertheless the terms $t^{-n}$, $n\in\Z_+$, in the heat expansion cause
a term $z^n \log z$ in the expansion of $\log\detz$. Insofar we cannot
confirm the claim made on the bottom of page 63 in \cite{BFK92}. Of course 
the Standing assumption \ref{sss.ZDER.SA} can be relaxed to
cover \cite{BFK92}'s result completely. If there is a pole at $0$
of the $\zeta$-function we define the regularized determinant as the 
finite part of $-\zeta'$. This is not a problem at all. However,
formulas become very clumsy.

\begin{proof}The result follows by combining
Lemmas \ref{p.Gamma}, \ref{p.zetanull}, and \ref{p.regint.expand}, 
Proposition \ref{Prop-Zetadeterminant}
and \Eqref{eq.Gamma.facts.1}, \eqref{eq.Gamma.facts.2}:
We apply \eqref{eq.zeta.det.asymp.heat.exp.1} to $L+z$ to obtain
\begin{equation*}
	\log \detz (L+z) = - \gamma\cdot A_{00}^H(L+z) - \regint_0^\infty t^{-1} 
	\tr(e^{-t(L+z)}) dt.
\end{equation*}
By Lemma \ref{p.zetanull} the first summand expand as
\begin{equation}\label{eq.AH.L+z}
	-\gamma \cdot A^H_{00}(L+z) = - \sum_{n=0}^{\infty} \Bl 
	\Res_{-n}\Gamma\Br\cdot 
	A^H_{-n,0}(L)\cdot \gamma \cdot z^n.
\end{equation}
To the second summand we apply Lemma  \ref{p.regint.expand} and obtain
\begin{equation}\label{eq.L+z.regint}
	-\regint_0^\infty t^{-1} \tr(e^{-t(L+z)})dt \sim_{\Re z\to \infty}  
	-\sum_{\ga,k} A^{H}_{\ga k}(L) \regint_0^{\infty} e^{-tz} t^{\ga-1} 
	\log^k t dt.
\end{equation}
To each regularized integral on the right hand side of \Eqref{eq.L+z.regint} 
we apply Lemma \ref{p.Gamma}:
\begin{align}
	-\regint_0^\infty e^{-tz} t^{\ga-1} \log^k t dt = & -\sum_{j=0}^k 
	\binom{k}{j}(-1)^j \Bl\Pf \partial^{k-j}\Gamma \Br(\ga)\cdot z^{-\ga} 
	\log^j z \label{eq.regint.tlogt.1}\\
	&+ \delta_{-\ga,\Z_+} \frac{(-1)^{k}\Bl\Res_{\ga}\Gamma\Br 
	}{k+1}z^{-\ga}\log^{k+1}z. \label{eq.regint.tlogt.2}
\end{align}
Due to  the assumption, $A^H_{\ga k}=0$ if $-\ga\in \Z_+$ and $k>0$ the last 
summand appears only with $k=0$. 

Combining \eqref{eq.AH.L+z} - 
\eqref{eq.regint.tlogt.2} 
we arrive at
\begin{align}
	\log&\detz (L+z) \sim_{\Re z\to \infty}\nonumber\\ 
	&- \sum_{\stackrel{\ga, 
	k}{-\ga\not 
	\in \Z_+}} A^H_{\ga k}(L) \sum_{ j=0}^{k} \binom{k}{j}(-1)^j \Bl \Pf 
	\partial^{k-j}\Br(\ga)\cdot z^{-\ga} \log^j z 
	\label{eq.asymp.log.zeta.det.4}\\
	&-\sum_{n=0}^\infty A^H_{-n, 0}(L)\cdot \Bl\gamma \cdot 
	(\Res_{-n}\Gamma)+(\Pf 
	\Gamma) (-n) \Br \cdot z^{-n}\label{eq.asymp.log.zeta.det.5}\\
	&+ \sum_{n=0}^\infty A^{H}_{-n, 0}(L)\cdot (\Res_{-n}\Gamma) \cdot z^{n} 
	\log 
	z.\label{eq.asymp.log.zeta.det.6}	
\end{align}

Substituting $(j,k)\rightarrow (k,j)$, \Eqref{eq.asymp.log.zeta.det.4} gives 
\eqref{eq.asymp.log.zeta.det.1}. Using \Eqref{eq.Gamma.facts.1} and 
\eqref{eq.Gamma.facts.2} one checks that
\begin{equation*}
	\Res_0 \Gamma = 1,\qquad
	\gamma\cdot \Res_0\Gamma + (\Pf \Gamma)(0) = 0,
\end{equation*}
and for $n>0$
\begin{equation*}
	(\Pf \Gamma)(-n) = (\Res_{-n}\Gamma) \cdot\Bl 
	L_n-\gamma\Br.
\end{equation*}
Thus \Eqref{eq.asymp.log.zeta.det.5}, \eqref{eq.asymp.log.zeta.det.6} 
together give \Eqref{eq.asymp.log.zeta.det.2} and 
\eqref{eq.asymp.log.zeta.det.3} completing the proof.
\end{proof}

We are now able to describe the Taylor expansion of $\log\det_{\zeta} 
(L+z)$ about $z=0$.

\begin{proposition}\label{p.taylor.exp.logdetzeta}
	Let $L$ satisfy the Standing assumptions \ref{sss.ZDER.SA}. Then 
	$\zeta(s;L+z)$ is holomorphic for $(s,z)$ in a neighbourhood of $(0,0)$ 
	and for $n\in\Z_+$, $n>0$,
	\begin{equation}\label{eq.deriv.zeta.z.zero}
		\frac{d^n}{dz^n}   \log\det_{\zeta}(L+z)\big|_{z=0} = 
		(-1)^{n-1}(n-1)!\Bl 
		\Pf_{s=n} \zeta(s;L) + \frac{A_{-n,0}^H}{(n-1)!} \cdot L_{n-1}\Br,
	\end{equation}
where $L_{n-1} = \sum\limits_{j=1}^{n-1} \frac{1}{j}$. Therefore, the Taylor 
expansion of $\log\det_{\zeta}(L+z)$ about $z=0$ reads
\begin{equation*}
	\begin{aligned}
		\log\det_{\zeta}(L+z)&= \log\det_{\zeta}(L)
		+ \sum_{n=1}^{\infty} \frac{(-1)^{n-1}}{n} \cdot 
		(\Pf_{s=n} \zeta(s;L))\cdot z^n\\
		&\qquad + \sum_{n=1}^\infty \frac{(-1)^{n-1}}{n!} \cdot 
		L_{n-1}\cdot A_{-n,0}^H \cdot z^n.
	\end{aligned}
\end{equation*}
Note that the last sum is in fact finite.
\end{proposition}

\begin{proof}
	By Hartogs's theorem \cite[Theorem 2.2.8]{Hor1990}, to prove that 
	the zeta function of $L+z$ is 
	holomorphic in a neighbourhood of $(0,0)$ it is enough 
	to verify that $\zeta(s;L+z)$ is 
	holomorphic in $s$ and $z$, separately. For fixed $z$ with 
	$|z|<\lambda_{\min}$, Lemma 
	\ref{p.zetanull} shows  that $s\mapsto \zeta(s;L+z)$ is regular near 
	$s=0$. For fixed $s$, $s$ small enough, one invokes the heat expansion to 
	see that 
	$z\mapsto \zeta(s;L+z)$ is regular near $z=0$. Therefore, $\zeta(s;L+z)$ 
	is holomorphic 
	in a neighbourhood of $(0,0)$ and we can commute the 
	derivative in $z$ with the 
	derivative in $s$ of $\zeta(s;L+z)$.
	Thus, let $n\in \Z_+$, $n>0$.
	\begin{equation*}
		\begin{aligned}
			\frac{d^n}{dz^n}   \log\det_{\zeta}(L+z)\big|_{z=0} &= 
			-\frac{d^n}{dz^n}  \Bl \frac{d}{ds} 
			\zeta(s;L+z)\big|_{s=0}\Br\big|_{z=0}\\
			&=-\frac{d}{ds} \Bl\frac{d^n}{dz^n}   \tr( 
			(L+z)^{-s})\big|_{z=0}\Br\big|_{s=0}\\
			&=-\frac{d}{ds} \Bl
			(-1)^{n-1}\Bl\prod_{j=1}^{n-1}(s+j)\Br\cdot 
			s \cdot\zeta(n+s;L)\Br \big|_{s=0}.
		\end{aligned}
	\end{equation*}
	Now we use \Eqref{eq.zeta.near.n} and obtain 
	\Eqref{eq.deriv.zeta.z.zero}. 
\end{proof}

\subsection{The heat expansion and the 
resolvent expansion}\label{Sub-sec-Heat-Resolvent}

The method of proof of Theorem \ref{Theorem-AsympExpLogZetaDet} in fact can 
be summarized in yet another Watson Lemma type expansion lemma. As an 
application we present a concise exposition of the equivalence between the 
heat and the resolvent expansion in full generality.

In this section we still assume the Standing assumptions 
\ref{sss.ZDER.SA}\footnote{Actually the assumption on the vanishing of 
$A^H_{-n,k}(L)$ for $n\in \Z_+$ and $k>0$ is not needed  and may be dropped 
in this section.}. For the self-adjoint operator $L$ the resolvent can be 
expressed in terms of the heat operator as
\begin{equation}\label{Eq-RelResolveHeatSemi}
(L+z)^{-1} := \int_{0}^{\infty} e^{-zt} e^{-t L} dt,
    \qquad \text{for} \qquad \Re{z}> 0,
\end{equation} 
and more generally,
\begin{equation}\label{Eq-RelnResolveHeatSemi}
(L+z)^{-N} := \frac{1}{(N-1)!}\int_{0}^{\infty} t^{N-1} e^{-zt} e^{-t L} 
dt, 
\qquad \text{for} 
\qquad
\Re{z}> 0.
\end{equation}
Under the Assumption \eqref{eq-Schatten}, we have for $N\geq p$
\begin{equation}\label{Eq-ResolHeat}
\tr \bl (L+z)^{-N} \br = \frac{1}{(N-1)!}\int_{0}^\infty t^{N-1} e^{-zt} 
\tr( 
e^{-t L}) dt.
\end{equation}
Assuming \Eqref{eq.heat.zero} the expansion of $\tr \bl (L+z)^{-N} \br$ as 
$\Re z \to \infty$ follows along the lines of the proof of the Theorem 
\ref{Theorem-AsympExpLogZetaDet} and the lemmas used therein. However, we 
think it is worth to single out the following regularized integral version of 
Watson's Lemma (\cf \cite[Pg. 133]{Wat1918}). 
\begin{lemma}\label{Lemma-Asymptotics}
	Let $q:(0,\infty)\to \C$ be a locally integrable function. Assume that as 
	$t\to 0+$ $q$ has an asymptotic expansion
	\begin{equation}\label{eq.asymp.lemma.q.1}
	q(t) \sim \sum_{\ga,k} A^q_{\ga k} t^{\ga-1} 
	\log^k t, \qquad \text{as}\qquad  t\to 0+,
	\end{equation}
	and that there exists $r_0>0$ such that
	\begin{equation}\label{eq.asymp.lemma.q.2}
		\int_1^\infty e^{-r_0 t} |q(t)| dt <\infty.
	\end{equation}
	Then we have an asymptotic expansion 
	\begin{align}
	\regint_{0}^\infty e^{-z t} &q(t) dt \sim 
	\sum_{\ga,k} A^q_{\ga k} \regint_{0}^\infty e^{-tz} t^{\ga -1} \log^k t \
	dt + O((\Re z)^{-N}), \;\; 
	\text{as}\;\; \Re z\to \infty, \label{eq.asymp.lemma.q.3}\\
	&\sim \sum_{\ga, k} (-1)^k \sum_{j\geq k} A^q_{\ga j} \binom{j}{k} (\Pf 
	\pl^{j-k} \Gamma)(\ga)\cdot z^{-\ga} \cdot \log^k z 
	\label{eq.asymp.lemma.q.4}\\
	&\qquad+\sum_{n=0}^{\infty} \sum_{k=0}^{k_n} A_{-n,k}^q \frac{(-1)^{k-1} 
	\Res_{-n}\Gamma}{k+1}\cdot z^n \cdot 
	\log^{k+1}z+O((\Re z)^{-N}),\label{eq.asymp.lemma.q.5}
	\end{align}
for any $N>0$.
\end{lemma}

\begin{remark}
(1) In  \Eqref{eq.asymp.lemma.q.1} we chose to write the 
		exponent of $t$ as  $\ga-1$ instead of $\ga$ to have a better 
		comparability to Lemmas \ref{p.Gamma}, \ref{p.zetanull}, 
		\ref{p.regint.expand} and Theorem \ref{Theorem-AsympExpLogZetaDet}.

\noindent (2) The notation $\sum\limits_{\ga,k}$ has the usual meaning as 
elaborated in Section \ref{ss.RIPF}, \cf also \Eqref{eq.heat.zero}.

\noindent (3) The expansion \eqref{eq.asymp.lemma.q.3} is valid as $\Re z\to 
\infty$ in particular it holds in any sector 
$\Lambda_{\epsilon}:=\setdef{z\in \C}{|\arg z|< \epsilon}$ for some 
$0<\epsilon<\pi/2$. Furthermore, in any such sector we have
\begin{equation*}
	O( (\Re z)^{-N} ) = O(|z|^{-N}),\qquad \text{as}\qquad |z|\to 
	\infty,
\end{equation*}
in $\Lambda_{\epsilon}$ for any $N>0$.
	
\end{remark}

\begin{proof}
	The proof follows the same scheme as the proof of Lemma 
	\ref{p.regint.expand} and of Theorem \ref{Theorem-AsympExpLogZetaDet}.
	
	First we split the integral into $\regint_{0}^1$ and $\int_1^\infty$. Due 
	to the assumption \eqref{eq.asymp.lemma.q.2} we find 
	\begin{equation*}
			\left| \int_1^{\infty} e^{-zt}q(t) dt\right| \leq e^{-(\Re 
			z-r_0)t} \int_1^\infty e^{-r_o t} |q(t)|dt = O\Bl e^{-(\Re z - 
			r_0)t} \Br = O((\Re z)^{-N}).
	\end{equation*}

	Fixing a large $K$ and writing
	\begin{equation*}
		\begin{aligned}
			q(t) &= \sum_{\Re \ga \leq K} \sum_{k=0}^{k_\ga} A_{\ga k}^q 
			t^{\ga 
		-1}  \log^k t + O(t^{K})=:q_0(t) + O(t^K),
		\end{aligned}
	\end{equation*}
	we note that 
	\begin{equation*}
		\begin{aligned}
			\left| \int_{0}^{1} e^{-zt} (q(t)-q_0(t)) dt \right|&\leq C\cdot  
			\int_{0}^1 
			e^{-(\Re z) t} t^{K} dt\\
			&= C\cdot \int_{0}^{(\Re z) t} e^{-u} u^{K} du \cdot (\Re 
			z)^{- K -1}= O\Bl (\Re z)^{- K -1}\Br.
		\end{aligned}
	\end{equation*}
	Therefore it remains to look at each individual summand which, as in the 
	proof of Lemma \ref{p.regint.expand}, contributes
	\begin{equation*}
		\regint_0^1 e^{-tz}\cdot t^{\ga-1}\log^k t \ dt = \regint_0^\infty 
		e^{-zt}\cdot 
		t^{\ga-1} 
		\log^k t \ dt - \int_1^\infty e^{-zt} \cdot t^{\ga-1}\log^k t\ dt.
	\end{equation*}
The second integral is $O\Bl e^{-(\Re z - \delta)}\Br$, $\Re z \to \infty$ 
for any $\delta>0$. Thus the expansion \eqref{eq.asymp.lemma.q.3} is proved.

As in the proof of Theorem \ref{Theorem-AsympExpLogZetaDet} we apply Lemma 
\ref{p.Gamma} to each regularized integral on the right hand side of 
\Eqref{eq.asymp.lemma.q.3}:
\begin{equation*}
	\begin{aligned}
		\regint_{0}^\infty& e^{-zt} t^{\ga-1}\log^k t \ dt\\ 
		=& \sum_{ j=0}^k 
		\binom{k}{j} (-1)^j (\Pf \pl^{k-j} \Gamma)(\ga)\cdot 
		z^{-\ga}\cdot\log^j z
		+ \delta_{-\ga,\Z_+} \frac{(-1)^{k-1}\Res_{\ga}\Gamma}{k+1} \cdot
		z^{-\ga}\cdot\log^{k+1}z.
	\end{aligned}
\end{equation*}
This gives \Eqref{eq.asymp.lemma.q.4} and \eqref{eq.asymp.lemma.q.5}.

 \end{proof}

\begin{proposition}
	Let $L$ be a self-adjoint bounded below operator in the Hilbert 
	space $\calH $ with resolvent in the $p^{th}$-Schatten class. Moreover, 
	assume 
	that 
	the trace of the heat semigroup has an asymptotic expansion 
	\Eqref{eq.heat.zero}. Let $N\in \N$, $N\geq p$. Then $\tr\Bl (L+z)^{-N} 
	\Br$ has the following asymptotic expansion as $\Re z \to \infty$
	\begin{equation*}
	\begin{aligned}
	\tr \bl(L+z)^{-N}\br \sim& 
	\sum_{\ga,k} (-1)^k \sum_{j\geq k} A^H_{\ga k}(L) \frac{(\Pf \pl^{j-k} 
	\Gamma)(\ga+N)}{(N-1)!}\cdot z^{-\ga-N}\cdot \log^k z\\
	&+\sum_{m=0}^\infty \sum_{k=0}^{k_m} A^H_{-m-N,k}(L) \frac{(-1)^{k-1} 
	\Res_{-m-N}\Gamma}{k+1} \cdot z^{m+N} \cdot \log^{k+1}z. 
	\end{aligned}
	\end{equation*}
\end{proposition}

\begin{proof}
	The result follows directly from the previous lemma and 
	\Eqref{Eq-ResolHeat}.
\end{proof}

On the other hand, we can write the heat semigroup as a contour integral in 
terms of the resolvent, 
\begin{equation}\label{Eq-HeatResol}
e^{-tL} = t^{1-N} \frac{(N-1)!}{2\pi \ i} \int_{C} e^{-t \mu}\cdot 
(L-\mu)^{-N} 
d\mu,
\end{equation}
where the contour $C$ is depicted in the Figure \ref{Figure1}. 
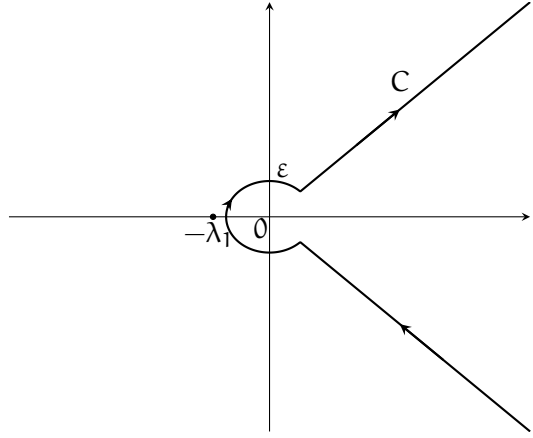
\begin{figure}[htb]
	\begin{tikzpicture}
		\begin{axis}[
			axis x line = center,
			axis y line = center,
			ymax = 6,
			ymin = -6,
			xmax = 6,
			xmin = -6,
			]
			
			\draw[thick] (axis cs:2,2)  -- (axis cs:6,6) ;
			\draw[>=stealth,<-,thick] (axis cs:3,3)  -- (axis cs: 0.701,0.701 
			) ;
			\draw[thick] (axis cs:0.701,-0.701)  -- (axis cs:4,-4) ;
			\draw[>=stealth,<-,thick] (axis cs:3,-3)  -- (axis cs:6,-6) ;
			
			\addplot [ samples=60, variable = \u,domain=45:165,thick]
			({cos(u)},{sin(u)});
			\addplot [>=stealth,<-,samples=60, variable = \u, 
			domain=150:315,thick]
			({cos(u)},{sin(u)});
			
			\draw[color=black] (-0.2,-0.3) node {$0$};
			\draw[color=black] (3,3.8) node {$C$};
            \draw[fill=black] (-1.3,0) circle[radius=1.0pt];
			\draw[color=black] (-1.4,-0.5) node {$-\lambda_1$};
			\draw[color=black] (0.3,1.3) node {$\varepsilon$};
		\end{axis}
	\end{tikzpicture}
	\caption{A graphic representation of the curve $C$.}\label{Figure1}
\end{figure}

Then one has, \cf \eg \cite[Lemma 2.2]{BL96}.
\begin{proposition}\label{Prop-RAHA}
Let $L$ be a self-adjoint bounded below operator in the Hilbert 
space $\calH$ with resolvent in the $p^{th}$-Schatten class.	
If the resolvent trace has the following asymptotic expansion
\begin{equation}\label{Eq-Res-AsympExp}
\tr\bl(L+z)^{-N}\br \sim \sum_{\ga,k} A^R _{\ga 
k}(L)
\cdot z^{-\alpha-N}\cdot \log^k z, \qquad {\text{as}}\qquad \Re z\to 
\infty,\; z\in 
\Lambda_{\epsilon},
\end{equation} 
for some $N\geq p$, then heat 
semigroup has the following asymptotic expansion as $t\to 0+$
\begin{equation*}
\tr(e^{-t L})\sim \sum_{\ga,k} (-1)^{k+1} (N-1)! \sum_{j=k}^{k_\ga} A^R_{\ga 
j}(L) 
\binom{j}{j-k} \Bl \Pf \partial^{j-k} 
\frac{1}{\Gamma} \Br(-\alpha-N) \cdot t^{\alpha}\  \log^k t.
\end{equation*}
\end{proposition}

We conclude this section observing that it is possible to obtain a result 
similar to Proposition \ref{Prop-Zetadeterminant} using the asymptotic 
expansion of the 
resolvent trace. This is done using the 
definition of the zeta function involving the resolvent trace, 
\begin{equation}
	\zeta(s;L) = \frac{\sin\pi 
		s}{\pi}\cdot\frac{(N-1)!\cdot\Gamma(1+s-N)}{\Gamma(s)}\regint_0^\infty
		 z^{N-1-s}\cdot \tr\bl(L+z)^{-N}\br dz,
\end{equation} 
for some $N\geq p$ and $\Re s >N$. Here we used  \Eqref{Eq-HeatResol} and 
the definition of Euler Gamma function (\cf \cite[8.310-2]{GR}) to write the 
last equation.

\section{The Fredholm determinant and the 
$\zeta$-determinant}
\label{s.FDZD}

\subsection{Fredholm determinant of operators in some Schatten class}
\label{ss.FDDDS}

We first recall some facts on the Fredholm determinant. Under a discrete
dimension spectrum assumption we will derive our main result on the
relation between the Fredholm determinant and the zeta-regularized
determinant. Our standard reference for Fredholm determinants and trace class
operators are \cite{Sim77} and \cite{Sim2005}.

Recall from \cite[Chap. 3]{Sim77} that 
for a trace class operator $A$ in the Hilbert space $\calH$
\begin{equation}\label{eq.fredhold.det.1}
\begin{aligned}
		\det (I + z\cdot A) &= \sum_{k=0}^{\infty} \tr(\Lambda^k A)\cdot z^k 
		\\
		&=\prod_{n=1}^\infty (1+ z\cdot \lambda_n(A)),
\end{aligned}
\end{equation}
where $\Lambda^k A$ is the induced operator on the exterior power $\Lambda^k 
\calH$ and $\{\lambda_n(A)\}$ is the summable sequence of eigenvalues of $A$. 
The latter identity is equivalent to Lindskii's theorem \cite[Theorem 
3.7]{Sim77}.

\subsubsection{Standing assumptions}\label{sss.FDER.SA} During this 
subsection, we assume that 
we are given a self-adjoint operator $L$ satisfying \Eqref{eq-BoundedBelow} 
and \eqref{eq-Schatten}. Additionally, we assume from now on that $L$ is 
invertible. Denote by $\lambda_1 \leq \lambda_2 \leq \ldots $ the sequence of 
eigenvalues of $L$ with multiplicities. If $p=1$ then $L^{-1}$ is trace class 
and \eqref{eq.fredhold.det.1} implies
\begin{equation}\label{Eq-FredholmDeterminant}
	\det (I + L^{-1}) = \prod_{q=1}^{\infty} (1+\gl_q^{-1}).
\end{equation}

Due to the fact that $L^{-1}$ is trace class the right hand side is an 
absolutely 
convergent
product.


\begin{proposition}\label{Prop-DerivativeFredholmDet}
Under the Standing assumptions \ref{sss.FDER.SA}, $z\mapsto \det(I + z\cdot 
L^{-1})$ is an entire function of order $1$ with zeros exactly in $-\spec L$. 
The order of the zeros $-\gl$, $\gl \in \spec L$, equals the multiplicity of 
the eigenvalue $\gl$. Moreover, 
\begin{equation}\label{Eq-DerivativeFredholdDet}
	\frac{d}{dz} \log \det (I+z\cdot L^{-1})=  \tr((L+z)^{-1}).
\end{equation}
\end{proposition}
\begin{proof}
The first claim is standard, \cf \cite[Chap. 3]{Sim77}, as the product 
expansion
\[
\det (I + z\cdot  L^{-1}) = \prod_{q=1}^{\infty} (1+z\cdot\gl_q^{-1})
\]
is a canonical Weierstra{\ss}-product of an entire function
of genus $0$. Therefore, the 
logarithmic
derivative may be taken term by term and the claim follows.
\end{proof}

It was shown in \cite[Section 6]{Sim77} that the previous
construction may be lifted to operators whose resolvent
is in the $p^{th}$-Schatten class. In this case one obtains a canonical
Weierstra{\ss}--product of higher genus. 
\begin{definition}\label{p.Fred.Det.N}
Suppose that $L^{-1}\in \B^p(\calH)$. Then for $N+1\geq p$ the
\emph{$N+1$-Fredholm determinant} of $I+L\ii$ is defined by
\begin{equation*}
\begin{aligned}
\det_{N+1}(I+L^{-1})&:= \prod_{q=1}^{\infty} \Bl (1+\ \gl_q^{-1}) \cdot 
\exp\Bl 
\sum_{k=1}^{N} (-1)^k\frac{\gl_q^{-k}}{k} \Br\Br\\
 &= \det \Bl(I+L^{-1}) 
\cdot \exp \Bl \sum_{k=1}^{N} (-1)^k \frac{L^{-k}}{k}\Br \Br.
\end{aligned}
\end{equation*}
\end{definition}
Note that for $p=1$, $\det_1(I+L^{-1}) = \det (I+L^{-1})$. The next theorem 
summarizes the main properties of the $N+1$-Fredholm determinant. 

\begin{theorem}\label{t.derivativeFD}
Let $L$ be a bounded below invertible self-adjoint operator with
$L^{-1}\in  \B^{p}(\calH)$. Then for $N+1\geq p$ the function $z\mapsto
\det_{N+1}(I + z\cdot L^{-1})$ is an entire function of order at most
$N+1$.  Moreover, $f_{N+1}(z):= \log\det_{N+1}(I + z\cdot L^{-1})$ satisfies
\begin{equation}
\frac{d}{dz} f_{N+1}(z) = (-z)^{N} \tr (L^{-N}(L+z)^{-1}),
\end{equation}
\begin{equation}\label{eq.LFDD.1}
f_{N+1}(0) = f'_{N+1}(0) = \ldots = f^{N}_{N+1}(0) = 0,
\end{equation}
and
\begin{equation}\label{eq.LFDD.2}
\frac{d^{N+1}}{dz^{N+1}} f_{N+1}(z) = (-1)^N N! \tr 
( (L+z)^{-(N+1)}).
\end{equation}
\end{theorem}

\begin{proof}
If follows from the very definition that $f_{N+1}(z)$ is the logarithm
of a canonical Weierstra{\ss}-product of genus $N$, see also
\cite[Theorem 3.3 and Lemma 6.1]{Sim77}. Consequently, the logarithmic
derivate is the compactly convergent sum of the logarithmic derivatives of
the individual factors:
\begin{equation*}
  \frac{d}{dz} \log \det_{N+1}(I+z \cdot L^{-1}) =
     \sum_{q=1}^{\infty}\left( (\gl_q +  z)^{-1} +
     \sum_{k=0}^{N-1} (-1)^{k+1} \Bl\frac{z}{\gl}\Br^k \cdot \gl_q^{-1}
     \right).
\end{equation*}
The sum inside equals, for $z\not = \gl_q$,
\begin{equation*}
	-\frac{1}{\gl_q} \frac{\Bl- \frac{z}{\gl_q} \Br^N - 
	1}{-\frac{z}{\gl_q}-1} = \frac{\Bl- \frac{z}{\gl_q} \Br^N-1}{z+\gl_q},
\end{equation*}
hence
\begin{equation*}
\begin{aligned}
\frac{d}{dz} \log \det_{N+1} (I + z\cdot L^{-1}) &
&= (-z)^{N} \tr(L^{-N}(L+z)^{-1}),
\end{aligned}
\end{equation*}
thus \Eqref{eq.LFDD.1} and \eqref{eq.LFDD.2} are proved.
\end{proof}

It is a direct consequence of the last theorem that if $L$ has 
discrete dimension spectrum then the Fredholm determinant of $L$  has an 
asymptotic expansion, as $\Re z\to \infty$, that can be differentiated. On 
the other side, if the 
Fredholm determinant of $L$ has an asymptotic expansion, as $\Re z\to 
\infty$, that can be 
differentiated then by Theorem \ref{t.derivativeFD} and Propositions 
\ref{Prop-DDS-HAE} and \ref{Prop-RAHA}, $L$ has discrete dimension spectrum. 
We summarize this discussion in the next corollary (compare with \cite[Lemma 
2.2]{Spr06}).

\begin{corollary}\label{Corollary-DiscDimenSpeclogFred}
	Let $L$ be a self-adjoint bounded below operator in the Hilbert 
	space $\calH$ satisfying \Eqref{eq-BoundedBelow} and \eqref{eq-Schatten}.
	For $L$ to have discrete dimension spectrum it is necessary and 
	sufficient that the logarithm of the Fredholm determinant has an 
	asymptotic expansion, as $\Re z \to \infty$, that can be differentiated.
\end{corollary}

The explicit knowledge of the coefficients of the 
asymptotic expansion 
of the Fredholm determinant as $\Re z \to \infty$ implies the explicit 
knowledge of the coefficients of the 
heat expansion as $t\to 0+$. However, the converse of this
implication is not true, \eg we will see below that the asymptotic expansion 
of the Fredholm determinant contains $\log\det_{\zeta} L$ which is a 
``global'' invariant of the heat trace.

\subsection{Asymptotic expansion of the Fredholm determinant}
\label{ss.RZDAE}
In this section we assume that $\calH$ is a separable Hilbert space and $L$ 
is an invertible
self-adjoint operator satisfying \Eqref{eq-BoundedBelow} and 
\eqref{eq-Schatten}.

Now we come to one of our main results.

\begin{theorem}\label{Theo-Main}
Let $L$ be an invertible self-adjoint bounded below operator in the 
Hilbert space
$\calH$ satisfying \Eqref{eq-BoundedBelow} and \eqref{eq-Schatten}. 
Furthermore
assume that $L$ has discrete dimension spectrum as defined in Def. \plref{p.DDS}
and that the $\zeta$-function $\zeta(s;L)$ of $L$ is regular at 
$0$ and has
at most simple poles in $s=1, \ldots, p-1$. Then
\begin{equation}
  \detz (L + z) = \exp\Bl \sum_{j=0}^{p-1} \frac{z^j}{j!} \cdot
  \frac{d^j}{dz^j} 
  \log \detz (L+z) |_{z=0} \Br \cdot \det_{p}(I + z \cdot L^{-1} ).
\end{equation}
\end{theorem}
\begin{remark}
Recall that the assumptions on the regularity at $0$ and the simple poles
in the positive integers are equivalent to the assumption that
in the heat expansion $A_{\ga, k}^H(L)=0$ for $-\ga\in\Z_+$ and $k>0$, \ie we 
are under the Standing assumptions \ref{sss.ZDER.SA}.
\end{remark}
\begin{proof}
	By Proposition \ref{p.taylor.exp.logdetzeta}, $\zeta(s;L+z)$ is 
	holomorphic for $(s,z)$ in a neighbourhood of $(0,0)$. Again we commute 
	the derivative in $z$ with the derivative in $s$ to obtain, 
	\begin{equation*}
	\begin{aligned}
	\frac{d^p}{dz^p} \log \det_{\zeta}(L+z) &= - \frac{d^p}{dz^p}\Bl 
	\frac{d}{ds} \zeta(s;L+z) |_{s=0}\Br
	&= - \frac{d}{ds}\Bl 
	\frac{d^p}{dz^p} \zeta(s;L+z) \Br|_{s=0}.
	\end{aligned}
	\end{equation*}
	This implies that
	\begin{equation*}
	\begin{aligned}
	-\frac{d}{ds}\frac{d^p}{dz^p} \zeta(s;L+z)&= \sum_{\gl \in \spec L} 
	(-s-1)\ldots (-s-p+1)(\gl+z)^{-s-p}\\
	&- (-s) \sum_{\gl \in \spec L} \frac{d}{ds} \Bl (-s-1)\ldots 
	(-s-p+1)(\gl+z)^{-s-p} \Br.
	\end{aligned}
	\end{equation*}
	Therefore, using  Theorem \ref{t.derivativeFD} 
	we obtain
	\begin{equation*}
	\begin{aligned}
	\frac{d^p}{dz^p} \log \det_{\zeta}(L+z) &= (-1)^{p-1}(p-1)! \tr 
	(L+z)^{-p}\\
	&= \frac{d^p}{dz^p} \log \det_p (I + z\cdot L^{-1}).
	\end{aligned}
	\end{equation*}
	Now the result follows from the series representation of the function 
	$\log\det_{\zeta}(L+z)$ at $z=0$.
\end{proof}

The value at zero is the best we can obtain from the heat expansion (\cf 
Proposition \ref{Prop-Zetadeterminant}). The 
$\zeta$-regularized determinant will not only depend
on the heat expansion. However, the 
asymptotic expansion of the 
Fredholm determinant delivers sufficient information to obtain the 
derivative of the $\zeta$-regularized determinant. In 
\cite{Fri}, the constant term of the 
asymptotic expansion of the Fredholm determinant as $\Re z \to \infty$ is 
determined for an elliptic differential 
operator, and it is 
equal to minus the logarithm of the $\zeta$-regularized determinant of the 
operator. 
A similar result for a self-adjoint operator in Hilbert 
space with trace class resolvent is proved in \cite[pag. 3439]{HLV17}. 
Theorem \ref{Theo-Main} provides 
an extension of these results as we see in the next corollary.


\begin{corollary}\label{Cor-AEFD}
	Let $L$ an invertible self-adjoint bounded below operator in 
	the Hilbert 
	space
	$\calH$ satisfying \Eqref{eq-BoundedBelow} and \eqref{eq-Schatten}. 
	Furthermore
	assume that $L$ has discrete dimension spectrum as defined in Def. 
	\plref{p.DDS}
	and that the $\zeta$-function of $L$ is regular 
	at $0$ and 
	has
	at most simple poles in $s=1, \ldots, p-1$.
	Then the logarithm of the Fredholm determinant of $I+z\cdot L^{-1}$ has 
	an 
	asymptotic expansion as follows
	\begin{align*}
			\log&\det_p (I + z\cdot L^{-1})\nonumber\\ 
			&\sim \sum_{\ga,k} A_{\ga k}^F(L) 
			\cdot 
			z^{-\ga}\cdot \log^k z,\qquad\text{ as }\qquad \Re z \to 
			\infty\nonumber\\
			&\sim -\sum_{\stackrel{\ga, k}{-\ga\not \in \Z_+}} (-1)^k 
			\sum_{j\geq k} A_{\ga j}^H(L) \binom{j}{k}\cdot (\Pf 
			\pl^{j-k}\Gamma)(\alpha) \cdot z^{-\alpha}\log z\\
			&\;\;\;\; - \log\det_{\zeta}(L)+ A_{00}^H(L) \cdot \log z\\
			&\;\;\;\; + \sum_{n=1}^{\infty} A^H_{-n,0}(L)\cdot 
			\frac{(-1)^n}{n!} \cdot z^n\log z 
			+ \sum_{n=p}^{\infty} A^H_{-n,0}(L)\cdot 
			\frac{(-1)^{n+1}}{n!}\cdot L_n \cdot z^n\\
			&\;\;\;\; + \sum_{n=1}^{p-1} \frac{(-1)^{n+1}}{n} \cdot 
			\Bl\frac{A^H_{-n,0}}{n!}  - 
			(\Pf_{s=n}\zeta(s;L))\Br\cdot z^n, \qquad
			\text{as} \qquad \Re z \to 
			\infty.
	\end{align*}
	In particular\footnote{The first two equalities of 
	\Eqref{Eq-Fredholm00term} and \eqref{Eq-Fredholm01term} were proved in 
	\cite[Proposition 
	2.6]{Spr06} assuming the asymptotic expansion of $\log\det_p(I+z\cdot 
	L^{-1})$ as $z\to \infty$.},
	\begin{align} 
			A_{00}^F(L) &= -\log\det_{\zeta}(L) = \;\gamma\cdot  
			A_{00}^H(L) + 
			\regint_{0}^\infty t^{-1} \tr (e^{-t L}) dt, 
			\label{Eq-Fredholm00term}\\
			A_{01}^F(L) &=\; \zeta(0;L) =\; A_{00}^H(L). 
			\label{Eq-Fredholm01term}
	\end{align}
\end{corollary}

\begin{proof}
	From Theorem \ref{Theo-Main} we have that
	\begin{equation*}
		\log\det_{p}(I+z\cdot L^{-1}) = \log\det_{\zeta}(L+z) - 
		\sum_{j=0}^{p-1} 
		\frac{z^j}{j!} \cdot
		\frac{d^j}{dz^j} 
		\log \detz (L+z) |_{z=0}.
	\end{equation*}
Now we apply Theorem \ref{Theorem-AsympExpLogZetaDet}, Proposition 
\ref{p.taylor.exp.logdetzeta} and obtain the asymptotic expansion of the 
Fredholm determinant as $\Re z \to \infty$.
\end{proof}

%

\bibliography{ZFdeterminants}
\bibliographystyle{amsalpha}


\end{document}